\documentclass{amsart}

\usepackage{amsmath,amssymb,amsfonts,amsthm}
\usepackage[all]{xy}

\title[]{A note on automorphisms and birational transformations of holomorphic symplectic manifolds}

\author{Samuel Boissi\`ere}

\address{Laboratoire J.A.Dieudonn\'e UMR CNRS 6621,
         Universit\'e de Nice Sophia-Antipolis, Parc Valrose, F-06108 Nice}

\email{Samuel.Boissiere@unice.fr}

\urladdr{http://math.unice.fr/$\sim$sb/}

\author{Alessandra Sarti}

\address{Laboratoire de Math\'ematiques et Applications, UMR CNRS 6086,
			Universit\'e de Poitiers, T\'el\'eport 2, Boulevard Marie et Pierre Curie,
			F-86962 Futuroscope Chasseneuil}
			
\email{sarti@math.univ-poitiers.fr}

\urladdr{http://www-math.sp2mi.univ-poitiers.fr/$\sim$sarti/}

\date{\today}



\newtheorem{theorem}{Theorem}

\newtheorem{lemma}{Lemma}
\newtheorem{proposition}{Proposition}

\newtheorem{remark}{Remark}

\newtheorem{question}{Question}

\DeclareMathOperator{\Aut}{Aut}
\DeclareMathOperator{\Bir}{Bir}
\DeclareMathOperator{\GL}{GL}
\DeclareMathOperator{\Orth}{O}
\DeclareMathOperator{\Mon}{Mon}
\DeclareMathOperator{\Pic}{Pic}
\DeclareMathOperator{\id}{id}

\DeclareMathOperator{\NS}{NS}
\DeclareMathOperator{\Tr}{T}

\DeclareMathOperator{\rank}{rk}
\DeclareMathOperator{\Pef}{Pef}
\DeclareMathOperator{\Hdg}{Hdg}
\DeclareMathOperator{\Ker}{Ker}
\DeclareMathOperator{\im}{Im}
\DeclareMathOperator{\Exc}{Exc}

\newcommand{\Pex}{\mathcal{P}ex}

\newcommand{\cA}{\mathcal{A}}
\newcommand{\cC}{\mathcal{C}}
\newcommand{\cD}{\mathcal{D}}
\newcommand{\cK}{\mathcal{K}}
\newcommand{\cL}{\mathcal{L}}
\newcommand{\cO}{\mathcal{O}}
\newcommand{\cS}{\mathcal{S}}

\newcommand{\IC}{\mathbb{C}}

\newcommand{\IR}{\mathbb{R}}
\newcommand{\IZ}{\mathbb{Z}}

\newcommand{\IP}{\mathbb{P}}

\newcommand{\kS}{\mathfrak{S}}

\newcommand{\Peftr}{\Pef_{\rm{tr}}}

\newcommand{\ie}{{\it i.e. }}

\newcommand{\loccit}{{\it loc. cit.}}

\begin{document}

\begin{abstract} 
We give a necessary and sufficient condition for an automorphism of the Hilbert scheme of points on a K3 surface (non necessarily algebraic) to be induced by an automorphism of the surface. We prove furthermore that the group of birational transformations of a projective irreducible holomorphic symplectic manifold is finitely generated.
\end{abstract}

\maketitle

\section*{Introduction}

Compact irreducible holomorphic symplectic manifolds are higher dimensional analogues of K3 surfaces and share many of their well-known properties. Their second integral cohomology space carries a natural weight-two Hodge structure and, due to a result of Beauville~\cite{Beauvillec1nulF}, it can be endowed with a natural non-degenerate quadratic form generalizing the intersection pairing of a K3 surface. 

If $X$ is a compact irreducible holomorphic symplectic manifold, the group $\Aut(X)$ of its biholomorphic automorphisms is discrete and any automorphism $f\in\Aut(X)$ induces by pull-pack an isometry of $H^2(X,\IZ)$ for the quadratic form $q_X$, yielding a natural map:
$$
\Phi\colon\Aut(X)\longrightarrow \Orth(H^2(X,\IZ),q_X), f\mapsto f^*.
$$
It is an important issue to understand to what extent an automorphism of $X$ is determined by the isometry induced on $H^2(X,\IZ)$.

If $X$ is a K3 surface, the strong Torelli theorem of Burns--Rapoport~\cite{BurnsRapoport} gives a precise answer: the map $\Phi$ is injective and for every isometry $\phi\in\Orth(H^2(X,\IZ),q_X)$, which is an isomorphism of integral Hodge structures and maps a K\"ahler class of $X$ to a K\"ahler class, there exists an automorphism $f\in\Aut(X)$ such that $f^*=\phi$. Using this theorem and lattice theory results, Nikulin~\cite{Nikulin} obtained an essentially complete understanding of the finite abelian automorphism groups of K3 surfaces. These results have been extended to non abelian groups by Mukai~\cite{Mukai}, and to infinite order automorphisms by McMullen~\cite{McMullen}.
In higher dimension, Debarre~\cite{Debarre} and Namikawa~\cite{Namikawa} provided counter-examples to various analogues of the global Torelli theorem. However, Verbitsky~\cite{Verbitsky} and Markman~\cite{MarkmanTorelli} recently proved the following version of the Global Torelli theorem.  Denote by $\Mon^2(X)\subset\GL(H^2(X,\IZ))$ the group of monodromy operators. The following version of the global Torelli theorem holds: for every $\phi\in\Mon^2(X)$ which is an isomorphism of integral Hodge structures and maps a K\"ahler class of $X$ to a K\"ahler class, there exists an automorphism $f\in\Aut(X)$ such that $f^*=\phi$.

It is difficult to construct interesting automorphisms of irreducible holomorphic symplectic manifolds. We refer to Beauville~\cite{BeauvilleInvol,BeauvilleKaehler}, Boissi\`ere--Nieper-Wi{\ss}kirchen--Sarti~\cite{BNWS}, Camere~\cite{Camere}, O'Grady~\cite{OGrady} and Oguiso--Schr\"oer~\cite{OS}   for some constructions concerning Hilbert schemes of two points on K3 surfaces, generalized Kummer varieties, Fano varieties of lines on cubic fourfolds, double covers of EPW sextics or O'Grady's resolutions of some moduli spaces of sheaves on K3 or abelian surfaces. The situation is particularly interesting when $X$ is the Hilbert scheme $S^{[n]}$ of $n$~points on a K3 surface $S$. In this case Beauville~\cite[Proposition~10]{BeauvilleKaehler} proved that the map $\Phi$ is injective, and Markman~\cite{MarkmanIntConst} obtained a very precise characterization of the group of monodromy operators that could, in the near future, help constructing new automorphisms. There is for the moment one standard way to construct automorphisms of $S^{[n]}$ by starting from automorphisms of $S$, yielding a morphism $\Aut(S)\to\Aut(S^{[n]})$ whose image consists of \emph{natural}~\cite{Boissiere} automorphisms of $S^{[n]}$. The first main result of this paper (Theorem~\ref{th:naturel}) is a characterization of the automorphisms obtained this way. It is easy to see that all these automorphisms leave invariant the class in $H^2(S^{[n]},\IZ)$ of the exceptional divisor of $S^{[n]}$, so it is natural to ask whether this condition is sufficient: we give in Theorem~\ref{th:naturel} a positive answer. The proof given here relies on some relations between the K\"ahler cones of $S$ and $S^{[n]}$ stated in Lemma~\ref{lem:Kal}. 
Note that a precise conjecture for the determination of the K\"ahler cone of $S^{[n]}$ is provided by Hassett--Tschinkel~\cite[Conjecture~1.2]{HT4}.

Denote by $\Bir(X)$ the group of bimeromorphic transformations of an irreducible holomorphic symplectic manifold $X$. If $X$ is non-projective, Oguiso~\cite[Theorem~1.5]{OguisoBir},~\cite[Theorem~1.4]{OguisoTits},~\cite[Theorem~1.6]{OguisoEntr} proved that the groups $\Aut(X)$ and $\Bir(X)$ are finitely generated, by showing that they are \emph{almost abelian of finite rank} (meaning that they are isomorphic to some $\IZ^r$, up to finite kernel and cokernel) and constructed projective examples where both groups are not almost abelian of finite rank. The second main result of this paper (Theorem~\ref{th:finitetype}) answers a question of Oguiso~\cite{OguisoTits} by showing that when $X$ is projective, the group $\Bir(X)$ is finitely generated. The proof relies on the recent proof of the global Torelli theorem by Verbitsky and Markman~\cite{MarkmanTorelli,Verbitsky}.

We thank Daniel Huybrechts and Manfred Lehn for their remarks, and the referee for helpful comments and references improving the paper.

\section{Basic tools on holomorphic symplectic manifolds}
\label{s:rappels}
Let $X$ be an irreducible holomorphic symplectic manifold of dimension $2n$ (with $n\geq 1$).
In the Hodge decomposition $H^2(X,\IC)=H^{2,0}(X)\oplus H^{1,1}(X)\oplus H^{0,2}(X)$ we put $H^{1,1}(X)_\IR:=H^{1,1}(X)\cap H^2(X,\IR)$. The  \emph{K\"ahler cone} $\cK_X$ is the open convex cone in $H^{1,1}(X)_\IR$ of classes which can be represented by a positive closed $(1,1)$-form.

We denote by $q_X$ the canonical \emph{Beauville--Bogomolov} \cite{Beauvillec1nulF} symmetric bilinear form on $H^2(X,\IZ)$ (we keep this notation also for its extension to $H^2(X,\IR)$ and $H^2(X,\IC)$ and for the associated quadratic form). This form is non-degenerate of signature $(3,b_2(X)-3)$ on $H^2(X,\IR)$ and it is such that $H^{1,1}(X)$ is orthogonal to $H^{2,0}(X)\oplus~H^{0,2}(X)$. The restriction of $q_X$ to $H^{1,1}(X)_\IR$ has signature $(1,b_2(X)-3)$ and for every K\"ahler class $\omega\in\cK_X$, we have $q_X(\omega)>0$.

Put $\cS_X:=\{\alpha\in H^{1,1}(X)_\IR\,|\,q_X(\alpha)>0\}$. The signature of $q_X$ implies that $\cS_X$ is the disjoint union of two open convex cones: the one containing the K\"ahler cone is denoted by $\cC_X$ and called the  \emph{positive cone}; the other component is denoted by $\cC'_X$ and we have the property: $x\in\cC_X \Leftrightarrow (-x)\in\cC'_X$. 

If $\omega\in\cK_X$ then $q_X(\omega,\cdot)$ is strictly positive on $\cC_X$ and for any effective divisor $D$ on $X$,
we have $q_X(\omega,[D])>0$ \cite[\S1.11]{HuybrechtsInvent}. In the particular case when $X$ is a K3 surface, this property characterizes the K\"ahler cone~\cite{BarthPetersVen}: if  $\omega\in\cC_X$ is such that $q_X(\omega,d)>0$ for any class $d$ of an effective divisor on $X$ such that $d^2=-2$, then $\omega\in\cK_X$.

Denote by $\cA_X\subset H^{1,1}(S)_\IR$ the \emph{ample cone}, generated by the first Chern classes  $c_1(L)$ of the ample line bundles $L$ on $X$. We have $\cA_X\subset\cK_X$. 

Let $\Peftr(X)\subset H^{1,1}(X,)_\IR$ be the set of \emph{pseudo-effective transcendental classes}, \ie  classes which can be represented by a closed positive $(1,1)$-current. Results of Debarre~\cite[\S3.3]{DebarreBrki}, Huybrechts~\cite[Proposition 1]{HuybrechtsInventErr} and Boucksom~\cite{Boucksom} show that it is a convex closed cone such that $\cK_X\subset~\Peftr(X)$ and $\overline{\cC_X}\subset\Peftr(X)$.

The \emph{N\'eron-Severi group} is $\NS(X):=H^{1,1}(X)_\IR\cap H^2(X,\IZ)$, of \emph{Picard number} $\rho(X):=\rank(\NS(X))$ and the \emph{transcendental lattice} $\Tr(X)$ is the orthogonal complement of $\NS(X)$ in $H^2(X,\IZ)$. We denote the signature of a lattice by $(n_1,n_2,n_3)$ where $n_1$ is the number of positive eigenvalues, $n_2$ of the zero eigenvalues and $n_3$ of the negative eigenvalues of the associated real quadratic form. There are three possibilities:
\begin{itemize}
\item \emph{hyperbolic} type: $\NS(X)$ is non--degenerate, of signature $(1,0,\rho(X)-1)$ and $\Tr(X)$ has signature $(2,0,b_2(X)-\rho(X)-2)$;
\item \emph{parabolic} type: $\NS(X)\cap\Tr(X)$ is of dimension $1$, $\NS(X)$ has signature $(0,1,\rho(X)-1)$ and $\Tr(X)$ has signature $(2,1,b_2(X)-\rho(X)-3)$ ;
\item \emph{elliptic} type: $\NS(X)$ is negative definite, of signature $(0,0,\rho(X))$ and $\Tr(X)$ has signature $(3,0,b_2(X)-\rho(X)-3)$.
\end{itemize} 
By Huybrechts \cite[Theorem 3.11]{HuybrechtsInvent}, $X$ is projective if and only if $\NS(X)$ is hyperbolic.

\section{ The Hilbert scheme of points on a K3 surface}
\label{s:hilbert}

Let $S$ be a K3 surface (not necessarily algebraic) and $n\geq 2$. We denote by $S^n$ the product of $n$ copies of $S$, $p_i\colon S^n\to S$ the projection onto the $i$-th factor, $S^{(n)}:=S^n/\kS_n$ the symmetric quotient of $S$, where the
symmetric group $\kS_n$ acts by permutation of the variables, $\pi\colon S^n\to S^{(n)}$ the quotient map, $\Delta$ the union of all the diagonals of $S^n$ and $D:=\pi(\Delta)$ its image in $S^{(n)}$. We denote by $S^{[n]}$ the 
\emph{ Hilbert scheme} (or \emph{Douady space} if $S$ is not algebraic) which parametrizes the analytic subspaces of $S$ of  dimension zero and length $n$. By Beauville~\cite{Beauvillec1nulF}, $S^{[n]}$~is an irreducible holomorphic symplectic manifold. The \emph{Hilbert--Chow morphism} (\emph{Douady--Barlet morphism} in the non algebraic case) $\rho\colon S^{[n]}\to S^{(n)}$ is projective and bimeromorphic, it is a resolution of singularities. We denote by $E:=\rho^{-1}(D)$ the exceptional divisor, which is irreducible.

There exists an injective morphism $\iota\colon H^2(S,\IC)\to H^2(S^{[n]},\IC)$ such that 
$$
H^2(S^{[n]},\IC)=\iota\left(H^2(S,\IC)\right)\oplus\IC[E]
$$
(we set $e:=[E]$) which is constructed as follows: for $\alpha\in H^2(S,\IC)$, there exists a unique $\beta\in H^2(S^{(n)},\IC)$ such that $\pi^*\beta=p_1^*\alpha+\cdots+p_n^*\alpha$ and we put  $\iota(\alpha):=\rho^*\beta$. The morphism $\iota$ is compatible with the Hodge decomposition~\cite[Proposition~6]{Beauvillec1nulF}. 
After normalisation, the form $q:=q_{S^{[n]}}$ satisfies $q(\iota(\alpha))=\alpha^2$ for $\alpha\in H^2(S,\IC)$, $q(e)=-8(n-1)$ and $e$ is orthogonal to $\iota\left(H^2(S,\IC)\right)$. There exists a class $\delta$ such that $2\delta=e$ and $H^2(S^{[n]},\IZ)=\iota\left(H^2(S,\IZ)\right)\oplus\IZ\delta$ (\loccit).

There exists a natural morphism of groups $-_n\colon\Pic(S)\to\Pic(S^{[n]})$ constructed as follows: for any line bundle $L\in\Pic(S)$, the line bundle $\bigotimes_{i=1}^np_i^*L$ projects to a line bundle $\cL$ on $\Pic(S^{(n)})$ and one defines $L_n:=\rho^*\cL$. By construction we have $c_1(L_n)=\iota(c_1(L))$. Denoting by $\left(\Pic(S)\right)_n$ the set of line bundles $L_n$, one has  $\Pic(S^{[n]})\cong\left(\Pic(S)\right)_n\oplus\IZ\cD$ with $\cD^2\cong\cO(-E)$ and $c_1(\cD)=-\delta$. The following lemma which compares the K\"ahler cones of $S$ and $S^{[n]}$ \textit{via} $\iota$ is the key to Theorem~\ref{th:naturel}.

\begin{lemma}\label{lem:Kal}\text{}
\begin{enumerate}
\item $\iota\left(\cC_S\right)\subset\cC_{S^{[n]}}$.

\item If $\omega:=\iota(\omega_0)+\lambda e\in\cC_{S^{[n]}}$, then $\omega_0\in\cC_S$. We have: 
$$
\iota\left(\cC_S\right)=~\cC_{S^{[n]}}\cap~\iota\left(H^{1,1}(S)_\IR\right).
$$

\item $\iota(\cK_S)\cap\cK_{S^{[n]}}=\emptyset$. If $\omega=\iota(\omega_0)+\lambda e\in\cK_{S^{[n]}}$ then $\lambda<0$ and $\omega_0\in\cK_S$. 
\end{enumerate}
\end{lemma}

\begin{proof}\text{}
\par{1.} Let $Z_n$ be the \emph{isospectral Hilbert scheme}, which is defined by Haiman~\cite{Haiman} as the reduced fiber product of $S^{[n]}$ with $S^n$ over $S^{(n)}$  and $\widetilde{Z_n}\to Z_n$ a resolution of singularities. We have a commutative diagram
$$
\xymatrix{\widetilde{Z_n}\ar[r]^g \ar[d]_f & S^n\ar[d]_\pi\\S^{[n]}\ar[r]^\rho & S^{(n)}}
$$
where $f$ and $g$ are surjective. For each $\alpha\in H^2(S,\IC)$, one sees immediately that $\iota(\alpha)$ is such that $f^*\iota(\alpha)=g^*(p_1^*\alpha+\cdots+p_n^*\alpha)$. For $\omega\in\cK_S$,  $\omega_n:=p_1^*\omega+\cdots+p_n^*\omega$ is a K\"ahler class of $S^n$ hence $g^*(\omega_n)$ is pseudo-effective, so $f^*(\iota(\omega))\in~\Peftr(\widetilde{Z_n})$. Since pseudo-effectivity is stable by considering the preimage (in the sense of currents, see  Debarre~\cite[\S3.1]{DebarreBrki}) by surjective maps between compact varieties --- a class is pseudo-effective if and only if its preimage is ---, this implies that $\iota(\omega)\in~\Peftr(S^{[n]})$, which excludes that $\iota(\omega)$ is contained in the connected component $\cC'_{S^{[n]}}$. We conclude that $\iota(\cC_S)\subset\cC_{S^{[n]}}$.

\par{2.} If $\omega=\iota(\omega_0)+\lambda e\in\cC_{S^{[n]}}$, we have $\omega_0^2-8\lambda^2(n-1)=q(\omega)>0$ hence $\omega_0^2>0$. Assume that $\omega_0\in\cC'_S$. Then $(-\omega_0)\in\cC_S$ so by the first assertion and by convexity we get $\omega+\iota(-\omega_0)=\lambda e\in\cC_{S^{[n]}}$ which is absurd, hence $\omega_0\in\cC_S$. 

\par{3.} If $\omega\in\cK_S$, then $q(\iota(\omega),e)=0$ hence $\iota(\omega)$ is not a K\"ahler class for $S^{[n]}$. If $\omega=\iota(\omega_0)+\lambda e\in\cK_{S^{[n]}}$, we have $0<q(\omega,e)=-8\lambda(n-1)$ since $e$ is the class of an effective divisor, which implies $\lambda<0$. By the second assertion we have $\omega_0\in\cC_S$. If $\omega_0\notin\cK_S$ there exists an effective divisor $D$ on $S$ such that $\omega_0\cdot[D]\leq 0$. Hence $\iota([D])$ is again the class of an effective divisor and we have $q(\omega,\iota([D]))=\omega_0\cdot[D]\leq 0$, contradiction.
\end{proof}

\begin{remark} The argument used for the first assertion has been communicated to us by Daniel Huybrechts. If $S$ is
 algebraic (hence projective), one can argue in a different way, avoiding the use of currents. Denote by $\Xi_n\subset~S\times~S^{[n]}$ the universal family and $p,q$ the respective projections on $S$ and  $S^{[n]}$; the morphism $p$ is flat. For $L\in\Pic(S)$, put $\psi(L):=\det(q_*p^*L)\in\Pic(S^{[n]})$. Beltrametti--Sommese~\cite[Theorem~A.1]{BelSom} proved that the map $\psi\colon \Pic(S)\to\Pic(S^{[n]})$ satisfies:
$$
\psi(L)=L_n\otimes\cO(-E).
$$
By Catanese--G{\"o}ttsche~\cite{CatGoe}, $\psi(L)$ is very ample if and only if  $L$ is $n$-very ample, which means that for any zero dimensional subscheme $\xi\subset S$  of length less than or equal to $n$, the canonical map $H^0(S,L)\to~H^0(S,L\otimes\cO_\xi)$ is surjective. In particular, if $L$ is ample then for $k$ big enough $L^k$ is very ample and $L^{kn}$ is $n$-very ample \cite[Lemma~0.1.1]{BelSom}, hence $\psi(L^{k n})$ is very ample. Since $S$ is projective, there exists a very ample line bundle $L$, hence $c_1(L)\in\cA_S$. We get: 
$$
c_1(\psi(L^{k n}))=k n\cdot\iota(c_1(L))-e\in\cA_{S^{[n]}},
$$
which excludes that $\iota(c_1(L))\in\cC'_{S^{[n]}}$ otherwise by convexity we would have $e\in\cC'_{S^{[n]}}$.
\end{remark}

\section{Classification of natural automorphisms}

Let $S$ be a K3 surface and $n\geq 2$. Any automorphism $\psi\in\Aut(S)$ induces an automorphism denoted $\psi^{[n]}\in\Aut(S^{[n]})$, called \emph{natural}, and the induced morphism ${\Aut(S)\to\Aut(S^{[n]})}$ is injective~\cite{Boissiere}. By the relation $\psi^{[n]}\circ\rho=\rho\circ\psi^{(n)}$, where $\psi^{(n)}$ is  the automorphism of $S^{(n)}$ induced by $\psi$, and by the fact that a natural automorphism  leaves globally invariant the exceptional divisor $E$, one obtains that the action of $\psi^{[n]}$ on $H^2(S^{[n]},\IZ)$ can be decomposed as $(\psi^{[n]})^*=(\psi^*,\id)$ in the decomposition $H^2(S^{[n]},\IZ)\cong H^2(S,\IZ)\oplus\IZ\delta$ (where $\iota$ is implicit). 

Let $f\in\Aut(S^{[n]})$. Since the exceptional divisor $E$ is ridig, the geometric property $f(E)=E$ (\ie $E$ is globally invariant) is equivalent to the algebraic property $f^*e=e$. We prove that this only condition characterizes the natural automorphisms.

\begin{theorem}\label{th:naturel}
Let $S$ be a K3 surface and $n\geq 2$. An automorphism $f$ of $S^{[n]}$ is natural if and only if it leaves globally invariant the exceptional divisor.
\end{theorem}

\begin{proof} The automorphism $f$ induces an isometry  $f^*$ of the lattice $\left(H^2(S^{[n]},\IZ),q\right)$ and if $f$ leaves globally invariant the exceptional divisor, we have $f^*(\delta)=\delta$. Since $\delta$ is orthogonal to $\iota\left(H^2(S,\IZ)\right)$ for the form $q$, $f^*$ stabilizes $\iota\left(H^2(S,\IZ)\right)$ hence it decomposes as $f^*=(\varphi,\id)$ where $\varphi$ is a Hodge isometry of the lattice $H^2(S,\IZ)$ since $\iota$ is compatible with the Hodge decomposition, as explained in the \S\ref{s:hilbert}.

Let $\omega\in\cK_{S^{[n]}}$. By Lemma \ref{lem:Kal} one can decompose it as  $\omega=\iota(\omega_0)+\lambda e$ with $\omega_0\in\cK_S$. Then  
$$
f^*(\omega)=\iota(\varphi(\omega_0))+\lambda e\in\cK_{S^{[n]}},
$$
so by Lemma~\ref{lem:Kal} again it follows that $\varphi(\omega_0)\in\cK_S$. By the global Torelli theorem for K3 surfaces, the effective Hodge isometry $\varphi$ is induced by a unique automorphism~$\psi$ of $S$ such that $\psi^*=\varphi$.

The natural automorphism $\psi^{[n]}$ of $S^{[n]}$ induced by $\psi$ satisfies $(\psi^{[n]})^*(\delta)=\delta$ and $(\psi^{[n]})^*_{|H^2(S,\IZ)}=\varphi$ so $(\psi^{[n]})^*=f^*$ on $H^2(S^{[n]},\IZ)$. By Beauville~\cite[Proposition~10]{BeauvilleKaehler} the map $\Aut(S^{[n]})\to \Orth(H^2(S^{[n]},\IZ))$ is injective, so $f=\psi^{[n]}$.
\end{proof}

\begin{remark} A similar classification result of automorphisms of generalized Kummer varieties is proven by Boissi\`ere--Nieper-Wi{\ss}kirchen--Sarti~\cite[Theorem 4.1]{BNWS}. The proof there uses different techniques since the 
generalized Kummer varieties admit non trivial automorphisms which act trivially on the second cohomology group \cite[Corollary 4.3]{BNWS}.
\end{remark}

\section{Applications and examples}

Let $S$ be a K3 surface and $n\geq 2$. For $f\in\Aut(S^{[n]})$, we define the \emph{index} of $f$ by:
$$
\lambda(f):=\frac{q(f^*(e),e)}{q(e)},
$$
in such a way that in $\NS(S^{[n]})$ we have $f^*(e)=\lambda(f) e + \iota(d)$ for some  $d\in\NS(S)$. Any natural automorphism has index~$1$ and for any natural automorphism $f$ and any automorphism ${g\in\Aut(S^{[n]})}$ we have:
$$
\lambda(f\circ g)=\lambda(g)=\lambda(g\circ f).
$$
In particular $\lambda(f\circ g \circ f^{-1})=\lambda(g)$ hence the index is invariant 
for the action of $\Aut(S)$ on $\Aut(S^{[n]})$ by conjugation. 

Boissi\`ere~\cite[Proposition~2]{Boissiere} proves by topological arguments that if ${\rho(S)=0}$ then all automorphisms of $S^{[n]}$ are natural. Theorem \ref{th:naturel} gives in particular an algebraic argument.
We study here the case $\rho(S)\geq 1$.

\begin{proposition}\label{hyperb}Let $S$ be a K3 surface and $n\geq 2$.
 If $S$ is of elliptic type, or of hyperbolic type with $\rho(S)=1$, then $f\in\Aut(S^{[n]})$ is natural if and only if $\lambda(f)=1$.
\end{proposition}

\begin{proof} If $\lambda(f)=1$, then $f^*(e)= e + \iota(d)$. Since $f^*$ is an isometry, we get $d^2=0$. If $S$ is of elliptic type, this implies that $d=0$ hence $f^*(e)=e$ and by Theorem \ref{th:naturel},  $f$~is a natural automorphism. The other implication is clear. In the hyperbolic case with $\rho(S)=1$, one has $\NS(S)\cong\IZ d$ with $d^2>0$ and the argument is similar.
\end{proof}

It would be interesting to understand if the invariant $\lambda\colon\Aut(S^{[n]})\to\IZ$ is enough to characterize the natural automorphisms. The previous result shows that 
this is the case for the generic algebraic K3 surface and in the elliptic case. It is not easy to find examples of non natural automorphisms of $S^{[n]}$. In fact there are, up to now, only two known examples. Take $S\subset\IP_3$ a generic K3 surface, containing no line.  Beauville~\cite[\S6]{BeauvilleKaehler} constructs an involution $i$ on $S^{[2]}$ as follows: for a reduced subscheme $\xi\in S^{[2]}$, the line $L$ through the two points in the support of $\xi$ cuts the surface $S$ in another length two subscheme $\xi'\in S^{[2]}$. One can prove that the obtained birational map $i\colon S^{[2]}\dashrightarrow S^{[2]}, \xi\mapsto\xi'$ extends to an automorphism. Denoting by $h$ the class of a hyperplane divisor, one computes that $i^*(\delta)=-3\delta+4h$ and $i^*(h)=-4\delta+3h$ (see Debarre \cite[Th\'eor\`eme 4.1]{Debarre} or Oguiso~\cite[Lemma 4.3]{OguisoRemark}). Since $\lambda(i)=-3$, this involution is not natural. The second example, which is inspired by the first one, is due to Oguiso (\loccit, see also Amerik~\cite{Amerik}). Consider a K3 surface $S$ admitting two embeddings as a quartic in $\IP_3$, given by two different very ample line bundles $H_1,H_2$, whose classes are denoted by $h_1,h_2$. Each embedding induces an involution $i_1,i_2$ on $S^{[2]}$ as before, satisfying the relations:
\begin{align*}
i_j^*h_j&=3h_j-4\delta,& i_j^*\delta&=-3\delta+2h_j,\quad j=1,2\\
i_1^*h_2&=8h_1-h_2-8\delta,& i_2^*h_1&=8h_2-h_1-8\delta.
\end{align*}
Consider the composition $i=i_1\circ i_2$. Then $i^*\delta=10h_1-2h_2-7\delta$ so $\lambda(i)=-7$ and $i$ is not natural (see also Oguiso~\cite[Lemma~4.6]{OguisoRemark}). One can easily see from these computations that $i$ is of infinite.

\section{The group of birational transformations}

Let $X$ be a projective irreducible holomorphic symplectic manifold. Using the global Torelli theorem of Verbitsky~\cite{Verbitsky} and Markman~\cite{MarkmanTorelli}, we prove that the group $\Bir(X)$ is finitely generated. Our proof is a generalisation of the argument of Sterk~\cite{Sterk} related to the automorphism group of projective K3 surfaces. 

Given $f\in\Bir(X)$, the correspondence by the closure in $X\times X$ of the graph of $f$ induces an automorphism $f^*\in\GL(H^2(X,\IZ))$. It is known (see Markman~\cite{MarkmanTorelli} and references therein) that $f^*\in\Orth(H^2(X,\IZ),q_X)$ and $f^*\in\Mon_{\Hdg}^2(X)$, where $\Mon_{\Hdg}^2(X)$ denotes the subgroup of $\Mon^2(X)$ of monodromy operators preserving the Hodge structure. Consider the group of birational transformations preserving the holomorphic two-form:
$$
\Bir^0(X):=\{f\in\Bir(X)\,|\, f^*\omega_X=\omega_X\}.
$$
Since every $f\in\Bir(X)$ induces a Hodge isometry of $H^2(X,\IZ)$ and $H^{2,0}(X)$ is generated by $\omega_X$, there is a character $\chi\colon\Bir(X)\longrightarrow \IC^*$
defined by $f^*\omega_X=\chi(f)\omega_X$ and ${\Bir^0(X)=\Ker(\chi)}$ is a normal subgroup of $\Bir(X)$.
 
\begin{lemma}\label{lem:firsttodo} Let $X$ be a projective irreducible holomorphic symplectic manifold. The quotient $\Bir(X)/\Bir^0(X)$ is a finite cyclic group.
\end{lemma}

\begin{proof} The argument is similar as in Sterk~\cite[Lemma~2.1]{Sterk} or Beauville~\cite[Proposition~7]{BeauvilleKaehler}.
Set $\Tr(X)_\IR:=\Tr(X)\otimes_\IZ\IR$ and $E:=(H^{2,0}(X)\oplus H^{0,2}(X))\cap H^2(X,\IR)$. There is an orthogonal decomposition:
$$
\Tr(X)_\IR=E\oplus (\Tr(X)_\IR\cap H^{1,1}(X)).
$$
Since $X$ is projective, $\Tr(X)$ is of signature $(2,0,\rho(X)-2)$. Note that $q_X$ is positive on $E$, so it is negative on the second space. The isometry $f^*\in\Orth(\Tr(X))$ induced by $f\in\Bir(X)$ preserves this decomposition and it is unitary on each space, so the eigenvalues of $f^*\in\Orth(\Tr(X)\otimes_\IZ\IC)$ have modulus $1$. Since they are algebraic integers, they are roots of the unity. In particular, the minimal polynomial of $\chi(f)$ is a cyclotomic polynomial $\Phi_n$ for some integer $n$. The polynomial $\Phi_n$ divides the characteristic polynomial of $f^*\in\GL(\Tr(X))$, so the Euler number of $n$ is smaller than or equal to $\rank(\Tr(X))$ and the possible values for $n$ are bounded. This shows that $\chi(\Bir(X))\subset\IC^*$ is a finite group.
\end{proof}

\begin{theorem}\label{th:finitetype} If $X$ is a projective irreducible holomorphic symplectic manifold, then $\Bir(X)$ is a finitely generated group.
\end{theorem}

\begin{proof} We follow Markman~\cite{MarkmanTorelli} to generalize Sterk's argument~\cite[Proposition~2.2]{Sterk}. By Lemma~\ref{lem:firsttodo}, it suffices to prove that $\Bir^0(X)$ is finitely generated. Consider the restriction morphism:
$$
\rho\colon\Mon^2_{\Hdg}(X)\longrightarrow \Orth(\NS(X)).
$$
Set $\cC_{\NS}:=\cC_X\cap\NS(X)$ and $\Orth^{+}(\NS(X)):=\{g\in\Orth(\NS(X))\,|\,g(\cC_{\NS})=\cC_{\NS}\}$. By Markman~\cite[Lemma~6.23]{MarkmanTorelli}, the group ${\Gamma:=\im(\rho)}$ is an arithmetic subgroup of finite index in $\Orth^+(\NS(X))$, so $\Gamma$ is finitely generated~\cite[Theorem~6.12]{BHC}. Define:
$$
\Gamma_{\Tr}:=\left\{g\in\Gamma\,|\,g_{|\Tr(X)}=\id\right\}.
$$
The group $\Gamma_T$ is an arithmetic subgroup of $\Gamma$, so it is again finitely generated. Denote by $\Pex$ the set of prime exceptional divisors of $X$ (\ie the set of reduced irreducible effective divisors of $X$ whose Gram matrix is negative definite). Set $W_{\Exc}$ the subgroup of $\Mon^2_{\Hdg}(X)$ generated by reflections by elements of $\Pex$ \cite[Definition~6.8, Theorem~6.18]{MarkmanTorelli}. Define:
\begin{align*}
\Gamma_{\Bir}&:=\{g\in\Gamma\,|\, g(\Pex)=\Pex\},\\
\Gamma_{\Tr,\Bir}&:=\{g\in\Gamma_{\Tr}\,|\, g(\Pex)=\Pex\}. 
\end{align*}
One has a semi-direct decomposition $\Gamma\cong \rho(W_{\Exc})\rtimes \Gamma_{\Bir}$ \cite[Lemma~6.23]{MarkmanTorelli}.
Since elements of $W_{\Exc}$ act trivially on $\Tr(X)$, one deduces that $\Gamma_{\Tr}\cong \rho(W_{\Exc})\rtimes \Gamma_{\Tr,\Bir}$ so $\Gamma_{\Tr,\Bir}$ is a quotient of $\Gamma_{\Tr}$, hence it is finitely generated. 

Denoting by $\Mon^2_{\Bir}(X)$ the subgroup of $\Mon^2_{\Hdg}(X)$ of monodromy operators induced by birational transformations of $X$, one has $\Ker(\rho)\subset \Mon^2_{\Bir}(X)$ and $\Gamma_{\Bir}\cong \Mon^2_{\Bir}(X)/\Ker(\rho)$ (see \cite[Lemma~6.23 and (3) p.26]{MarkmanTorelli}), so:
\begin{align*}
\Gamma_{\Tr,\Bir}&\cong\left\{g\in\Mon^2_{\Bir}(X)\,|\,g_{|\Tr(X)}=\id\right\}/\left\{g\in\Mon^2_{\Bir}(X)\,|\,g_{|\NS(X)}=\id\right\}\\
&\cong \left\{f\in\Bir(X)\,|\,f^*_{|\Tr(X)}=\id\right\}/\left\{f\in\Bir(X)\,|\,f^*_{|H^2(X,\IZ)}=\id\right\}.
\end{align*}
Since $X$ is projective, one has $\Bir^0(X)\cong\left\{f\in\Bir(X)\,|\,f^*_{|\Tr(X)}=\id\right\}$ and it is well-known that $\left\{f\in\Bir(X)\,|\,f^*_{|H^2(X,\IZ)}=\id\right\}$ is a finite group (see \cite[Proposition~9.1]{HuybrechtsInvent}): a birational transformation acting trivially on the second cohomology space leaves invariant a K\"ahler class, so extends to an automorphism by Fujiki's theorem and is an isometry for the Calabi--Yau metric uniquely associated to this K\"ahler class; the assertion follows then from the fact that the group of isometries of a compact Riemannian manifold is compact and $\Aut(X)$ is discrete. As a consequence, $\Bir^0(X)$ is finitely generated.
\end{proof}

As mentioned by Oguiso~\cite[Question~1.6]{OguisoTits}, it is not known whether $\Aut(X)$ is of finite index in $\Bir(X)$, so the following question remains open:

\begin{question}
For $X$ a projective irreducible holomorphic symplectic manifold, is the group $\Aut(X)$ finitely generated?
\end{question}

The key difference between understanding automorphisms and birational transformations of $X$ via various types of subgroups of $\Orth(\NS(X))$ is that, in the moduli space of marked irreducible holomorphic symplectic manifolds, two elements with the same period are always bimeromorphic \cite[Theorem~4.3]{HuybrechtsInvent} but, in contrary to the case of K3 surfaces, in higher dimension they are not automatically isomorphic. This also explains why constructing automorphisms of $X$ from the action on the lattice $H^2(X,\IZ)$ is much more difficult for higher dimensional irreducible holomorphic symplectic manifolds than for K3 surfaces.

\bibliographystyle{amsplain}
\bibliography{BiblioAutHilbK3}

\end{document}